\documentclass[11pt]{article}
\usepackage[a4paper,margin=1in]{geometry}
\usepackage{amsmath,amssymb,amsthm,mathtools}
\usepackage{enumitem,microtype}
\usepackage{tikz}
\usetikzlibrary{arrows.meta,positioning}
\usepackage{setspace}
\usepackage[colorlinks=true,linkcolor=blue,citecolor=blue,urlcolor=blue]{hyperref}
\newtheorem{theorem}{Theorem}[section]
\newtheorem{proposition}[theorem]{Proposition}
\newtheorem{lemma}[theorem]{Lemma}

\theoremstyle{definition}
\newtheorem{assumption}[theorem]{Assumption}
\newtheorem{remark}[theorem]{Remark}
\newcommand{\R}{\mathbb R}
\newcommand{\loc}{\mathrm{loc}}
\newcommand{\dd}{\,\mathrm d}

\title{Radial Powered Mean Curvature Flow with Robin Boundary Conditions}
\author{Tianhong Pu$^\dag$, Xiaoqian Xin$^\dag$ and Lixia Yuan$^{\dag,*}$\\
$^\dag$ {\small Mathematics and Science College, Shanghai Normal University, Shanghai 200234, China.}\\
$^*$ {\small Corresponding author: Lixia Yuan  (yuanlixia@shnu.edu.cn)} 
}
\date{\today}

\begin{document}
\maketitle

\begin{abstract}
We study a positive radial graph evolving by $V=H^\alpha+b$, with $\alpha>0$ and $b<0$, in the unit cylinder and with $u_r(1,t)=u(1,t)$. We separate finite-time continuation estimates from the time-uniform interior estimates needed for convergence of normalized translates. On existing classical intervals we prove height and gradient bounds and an explicit conversion of velocity bounds into estimates for $H$, $u_r/r$, and $u_{rr}$. We derive the correct boundary defect for comparisons with finite-slope translators. Under independent interior velocity bounds and two-sided slope trapping, we prove compactness and identify every translated limit with the cup translator. We identify the radial zero-number comparison still required to establish that trapping. Stationary caps and power-law Robin laws are included.
\end{abstract}

\section{Introduction}

Let $B_1\subset\R^N$, $N\ge2$, and let $u=u(r,t)$ be a radial graph. With the upward normal convention,
\begin{equation}\label{eq:pde}
 u_t=(H^\alpha+b)\sqrt{1+u_r^2},\qquad
 H=\frac{u_{rr}}{(1+u_r^2)^{3/2}}+
 \frac{N-1}{r}\frac{u_r}{\sqrt{1+u_r^2}},
\end{equation}
and
\begin{equation}\label{eq:bc}
 u_r(0,t)=0,\qquad u_r(1,t)=u(1,t).
\end{equation}
At $r=0$ the equation is interpreted through the smooth radial lift to $\R^N$; in particular $H(0,t)=Nu_{rr}(0,t)$. We take compatible smooth positive data with $u_{0rr}>0$ and $(H[u_0]^\alpha+b)>0$.

The literature relevant to this problem includes the planar periodic waves of Lou \cite{Lou2009}, the one-dimensional unbounded-slope convergence theorem of Yuan--Zhao \cite{YuanZhao2022}, and the radial constant-angle theory of Lou--Yuan \cite{LouYuan2023}. In the unbounded-slope planar case, Lou--Wang--Yuan \cite{LouWangYuan2021} first obtain time-uniform interior gradient estimates by a zero-number argument and only then take time translates. The radial zero-number property needed here is available from Angenent \cite{Angenent1988} and Chen--Pol\'a\v cik \cite{ChenPolacik1996}.

The main logical point is the following. A bound on $[0,T]$ with a constant $C_T$ is enough for continuation at a prescribed finite time. It is not enough for convergence as $t\to\infty$. For convergence one needs, for every $R<1$ and $\tau>0$,
\begin{equation}\label{eq:uniform-needed}
 \sup_{t\ge\tau}\sup_{0\le r\le R}(|u_r|+|u_{rr}|+|u_t|)<\infty.
\end{equation}
The boundary is excluded because $u_r(1,t)=u(1,t)$ diverges on the cup branch.

\begin{figure}[htbp]
\centering
\begin{tikzpicture}[>=Latex,font=\small]
 \draw[->] (-.15,0)--(3.7,0) node[right] {$r$};
 \draw[->] (0,-.12)--(0,3.1) node[above] {$u$};
 \draw[thick,blue!65!black,domain=0:3.2,samples=80] plot (\x,{.3+.25*\x*\x});
 \coordinate (P) at (2,1.3); \fill (P) circle (1.4pt);
 \draw[->,thick] (P)--(2.65,1.95) node[right] {$T$};
 \draw[->,thick] (P)--(1.35,1.95) node[above left] {$\nu$};
 \draw[densely dashed,gray] (P)--(2,0) node[below,black] {$r_0$};
 \node[anchor=west] at (4.0,2.45) {$T=(1,u_r)/\sqrt{1+u_r^2}$};
 \node[anchor=west] at (4.0,1.80) {$\nu=(-u_r,1)/\sqrt{1+u_r^2}$};
 \node[anchor=west] at (4.0,1.05) {$\displaystyle\kappa_r=\frac{u_{rr}}{(1+u_r^2)^{3/2}}$};
 \node[anchor=west] at (4.0,.28) {$\displaystyle\kappa_\tau=\frac{u_r}{r\sqrt{1+u_r^2}}$};
\end{tikzpicture}
\caption{Meridian geometry and the decomposition $H=\kappa_r+(N-1)\kappa_\tau$. The illustrative meridian is $u=0.3+r^2/4$; the drawn tangent and normal have their exact directions at $r_0=2$.}
\label{fig:geometry}
\end{figure}
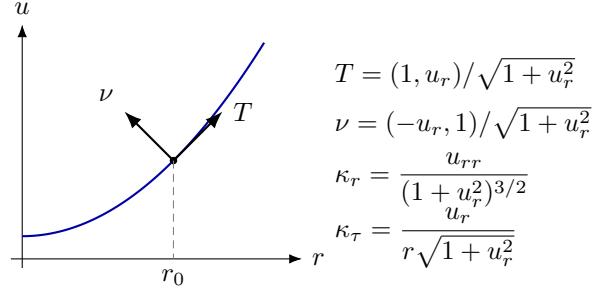

Put $B=(-b)^{1/\alpha}$. The following theorem states exactly what the convergence proof uses.
\begin{theorem}[Conditional cup convergence]\label{thm:main}
Assume the translating-profile properties in Assumption~\ref{ass:profiles}, and let $u$ be a global smooth solution of \eqref{eq:pde}--\eqref{eq:bc}. Suppose that for each $R<1$ and $\tau>0$ there are constants $P_{R,\tau},V_{R,\tau}$, independent of time, such that
\begin{equation}\label{eq:PV}
 0\le u_r(r,t)\le P_{R,\tau},\qquad 0<u_t(r,t)\le V_{R,\tau}quad(0\le r\le R, t\ge\tau),
\end{equation}
and suppose that for every sufficiently large finite slope $h$,
\begin{equation}\label{eq:slope-trap}
 \Phi_h'(r)\le u_r(r,t)\le\Phi_*'(r)quad(0\le r\le R, t\ge t_{h,R}).
\end{equation}
Then
\begin{equation}\label{eq:conv}
 u(r,s+t)-u(0,s)\longrightarrow\Phi_*(r)+c_*t
 \quad\text{in }C^{2,1}_{\loc}([0,1)\times\R).
\end{equation}
Consequently $u_t\to c_*$ and $u_{rr}\to\Phi_*''$ locally uniformly in $r<1$.
\end{theorem}

The estimate \eqref{eq:PV} and the slope trapping \eqref{eq:slope-trap} are not consequences of finite-time constants $C_T$. Their proof is the higher-dimensional analogue of Section~4 of \cite{YuanZhao2022} and Section~4 of \cite{LouWangYuan2021}; the exact radial boundary defect is displayed in Section~4 below.

\section{Translating profiles}

\begin{assumption}[Profiles from Lou--Yuan]\label{ass:profiles}
There are $c_*>0$ and $\Phi_*$ on $[0,1)$ satisfying
\begin{equation}\label{eq:TS}
 (H[\Phi_*]^\alpha+b)\sqrt{1+\Phi_*'^2}=c_*,\qquad
\Phi_*(0)=\Phi_*'(0)=0,
\end{equation}
with $\Phi_*',\Phi_*''>0$, $\Phi_*(1^-)<\infty$, and $\Phi_*'(r)\to\infty$ as $r\uparrow1$. For arbitrarily large $h$ there are finite-slope profiles $\Phi_h$ and speeds $c(h)$ with $\Phi_h'(1)=h$, $c(h)\uparrow c_*$, and $\Phi_h\to\Phi_*$ in $C^2([0,R])$ for each $R<1$. We also use the profile-ODE ordering from the same result: for every sufficiently large $h$ one can choose $q=q(h)<1$ so that the compatible profile $\rho_q$ in \eqref{eq:rhoq} satisfies
\[
 \Phi_h'(r)<q\Phi_*'(qr)<\Phi_*'(r)\qquad(0<r<1).
\]
\end{assumption}

These are the only translating-solution facts used from \cite[Theorems 1.3--1.4 and Section 2]{LouYuan2023}. In particular, the endpoint-height and infinite-slope properties are hypotheses here, not an automatic consequence of writing $B<N$ without fixing the curvature normalization.

For a smooth finite-slope translator $L_h=\Phi_h+c(h)t+A_h$, the interior equation is an equality. For a general increasing boundary law $u_r(1,t)=g(u(1,t))$, the lower-solution condition is exactly
\begin{equation}\label{eq:barrier-condition}
 g(\Phi_h(1)+A_h)\ge h.
\end{equation}
If comparison starts at $t=0$, one also needs $A_h\le\min_{[0,1]}(u_0-\Phi_h)$. For $g(s)=s$, these two inequalities can fail to have a common $A_h$ for a given $h$; the phrase ``after a vertical translation'' must therefore include both conditions.

\begin{lemma}[Singular upper comparison]\label{lem:upper}
If $u_0\le\Phi_*+C$ and a smooth solution exists on $[0,T]$, then
\begin{equation}\label{eq:upper}
 u(r,t)\le\Phi_*(r)+c_*t+C\quad(0\le r\le1, 0\le t\le T).
\end{equation}
\end{lemma}
\begin{proof}
For $\delta>0$ put $W_\delta=\Phi_*+c_*t+C+\delta(1+t)$. It is a strict supersolution in $r<1$. At an interior first contact with $u-W_\delta$, equality of first derivatives and nonpositivity of the Hessian difference imply $u_t\le F[W_\delta]<(W_\delta)_t$, a contradiction. At the outer boundary, a classical solution has finite $u_r$ on the fixed time interval, whereas $\Phi_*'(r)\to\infty$; consequently $u-W_\delta$ is strictly decreasing in a one-sided neighbourhood of $1$, so a positive boundary maximum cannot be the first contact. Letting $\delta\downarrow0$ proves the claim.
\end{proof}

\begin{figure}[htbp]
\centering
\begin{tikzpicture}[>=Latex,font=\small]
 \begin{scope}
  \fill[blue!8] (0,0) rectangle (3.2,2.1);
  \draw[blue!65!black,thick] (0,0) rectangle (3.2,2.1);
  \draw[->] (-.1,0)--(3.7,0) node[right] {$r$};
  \draw[->] (0,-.1)--(0,2.9) node[above] {$t$};
  \node[below] at (0,0) {$0$}; \node[below] at (3.2,0) {$1$};
  \node[left] at (0,2.1) {$T$};
  \node[align=center] at (1.6,1.1) {finite cylinder\\bounds $C_T$};
  \node at (1.6,-.8) {finite-time continuation};
 \end{scope}
 \begin{scope}[xshift=6.3cm]
  \fill[teal!8] (0,.45) rectangle (2.35,2.65);
  \draw[teal!70!black,thick,->] (2.35,.45)--(2.35,2.85);
  \draw[->] (-.1,0)--(3.7,0) node[right] {$r$};
  \draw[->] (0,-.1)--(0,2.9) node[above] {$t$};
  \draw[densely dashed,gray] (3.2,0)--(3.2,2.7);
  \draw[thick,teal!70!black] (0,1.05) rectangle (2.35,1.8);
  \node[below] at (2.35,0) {$R<1$}; \node[below] at (3.2,0) {$1$};
  \node[left] at (0,.45) {$\tau$};
  \node[font=\footnotesize] at (1.16,1.43) {$[s-L,s+L]$};
  \node at (1.6,-.8) {bounds independent of $s$};
 \end{scope}
\end{tikzpicture}
\caption{The domains of the two different estimates. The right-hand constants must remain bounded as the time window moves upwards. Its spatial range stays strictly inside $r=1$; the finite-time estimate on the left does not imply this uniform bound.}
\label{fig:barriers}
\end{figure}
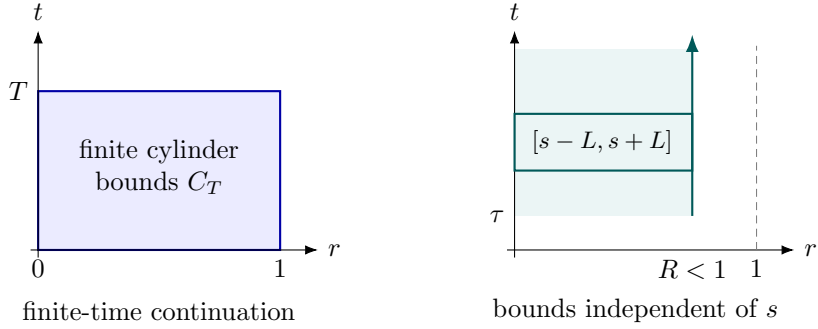

\section{Finite-time estimates and continuation}

Set $p=u_r$, $S=(1+p^2)^{1/2}$ and $w=u_t$. The useful identities are
\begin{equation}\label{eq:Fq}
 F_q=\frac{\alpha H^{\alpha-1}}{1+p^2},\qquad
 w_t=F_qw_{rr}+F_pw_r,
\end{equation}
with radial boundary conditions
\begin{equation}\label{eq:wbc}
 w_r(0,t)=0,\qquad w_r(1,t)=w(1,t).
\end{equation}
The second identity is the crucial difference from a fixed contact-angle problem, for which time differentiation gives $w_r(1,t)=0$.

\begin{lemma}[Finite-time first-order bounds]\label{lem:C1}
Let $m_0=\min_{[0,1]}F[u_0]>0$. On every existing classical interval in the positive-curvature branch,
\begin{equation}\label{eq:wlower}
 u_t\ge m_0,qquad H>B.
\end{equation}
If the upper comparison \eqref{eq:upper} holds, then for every $T<\infty$,
\begin{equation}\label{eq:C1T}
 0<u(r,t)\le U_T:=C_0+c_*T,qquad
 0\le u_r(r,t)\le P_T:=\max\{\|u_0'\|_\infty,U_T\}.
\end{equation}
\end{lemma}
\begin{proof}
The constant $m_0$ is a subsolution for the linear equation in \eqref{eq:Fq}: its interior residual is zero and its outer boundary residual is $-m_0<0$. The radial maximum principle applied to $w-m_0$ gives $w\ge m_0$; the centre is handled by the smooth radial lift. Since $w=(H^\alpha+b)S$ and $S\ge1$, \eqref{eq:wlower} gives $H>B$.

Integration in time gives $u\ge u_0+m_0t>0$. The upper bound is \eqref{eq:upper}. To justify positivity of $p$ without assuming preservation of convexity, use
\[
 \frac{p(r,t)}{\sqrt{1+p(r,t)^2}}
 =r^{1-N}\int_0^r H(\rho,t)\rho^{N-1}\dd\rho>0\quad(r>0).
\]
Differentiating \eqref{eq:pde} gives the exact equation
\begin{equation}\label{eq:p-detailed}
 p_t=F_qp_{rr}+F_pp_r
       -\frac{\alpha(N-1)H^{\alpha-1}p}{r^2}.
\end{equation}
At an interior positive maximum of $p$, the diffusion term is nonpositive and the last term is negative. At $r=0$, $p=0$, and at $r=1$, $p=u(1,t)\le U_T$. Applied to $p-P_T-\delta t$ and followed by $\delta\downarrow0$, the first-contact maximum principle gives \eqref{eq:C1T}. Its constant depends on $T$; it is not a time-uniform estimate.
\end{proof}

\begin{lemma}[Velocity-to-Hessian conversion]\label{lem:Hessian}
On a radial cylinder $[0,R]\times I$ on the positive-curvature branch, suppose $|u_r|\le P$ and $0\le u_t\le V$. Put $K=(V-b)^{1/\alpha}$ and $S_P=(1+P^2)^{1/2}$. Then
\begin{align}
 B\le H&\le K,\qquad
 0\le\frac{u_r}{r}\le\frac{S_PK}{N},\label{eq:Hbounds}\\
 |u_{rr}|&\le S_P^3K.\label{eq:rrbound}
\end{align}
The quotient at $r=0$ is the continuous radial value.
\end{lemma}
\begin{proof}
The equation gives $H=((u_t/S)-b)^{1/\alpha}$, hence $B\le H\le K$. With $z=u_r/S$, the radial identity is
\[
 z_r+\frac{N-1}{r}z=H,qquad
 \frac zr=r^{-N}\int_0^rH(\rho,t)\rho^{N-1}\dd\rho\le K/N.
\]
Since $u_r=Sz$, this gives the second inequality in \eqref{eq:Hbounds}. Finally,
\[
 u_{rr}=S^3\left(H-(N-1)\frac zr\right),
\]
which proves \eqref{eq:rrbound}, including at the centre by taking the integral limit.
\end{proof}

\begin{proposition}[What is sufficient for continuation]\label{prop:continuation}
Assume that, for each $T<\infty$, the existing solution satisfies the finite-time velocity upper bound
\begin{equation}\label{eq:finiteV}
 \sup_{0\le t<\min\{T,T_{\max}\}}\|u_t(\cdot,t)\|_{\infty}\le V_T<\infty.
\end{equation}
Then $u$, $u_r$, $u_{rr}$ are bounded on the same cylinder, and $F_q$ is bounded above and below by positive constants depending on $T$. With the corresponding oblique-boundary regularity theorem, this excludes a finite maximal time.
\end{proposition}
\begin{proof}
Use Lemma~\ref{lem:C1} for $u$ and $u_r$, and Lemma~\ref{lem:Hessian} for $H$ and $u_{rr}$. If $K_T=(V_T-b)^{1/\alpha}$ and $P_T$ is from \eqref{eq:C1T}, then
\[
 \frac{\alpha\min\{B^{\alpha-1},K_T^{\alpha-1}\}}{1+P_T^2}
 \le F_q\le\alpha\max\{B^{\alpha-1},K_T^{\alpha-1}\}.
\]
Standard continuation for the resulting uniformly parabolic equation with oblique boundary data applies. The proposition is conditional because \eqref{eq:finiteV} has not been proved here for the nonlinear Robin boundary condition.
\end{proof}

The common attempted proof of \eqref{eq:finiteV} by an exponential barrier for $w$ is circular unless the coefficient $F_p$ is already controlled: it contains $u_{rr}$. Likewise, for $Q=H^\alpha+b=u_t/S$ the differentiated boundary identity is
\begin{equation}\label{eq:Qbc}
 Q_r(1,t)=\left(1-\frac{u_r(1,t)u_{rr}(1,t)}{1+u_r(1,t)^2}\right)Q(1,t),
\end{equation}
so a $C^1$ estimate alone does not close the boundary maximum principle.

\subsection{Finite-time bookkeeping at the centre}

For later use we spell out how the apparent singular coefficients are removed. Let $z=p/S$. Radial regularity gives
\[
 p(r,t)=u_{rr}(0,t)r+O(r^3),\qquad z(r,t)=u_{rr}(0,t)r+O(r^3),
\]
and therefore
\[
 \frac{z(r,t)}r\longrightarrow u_{rr}(0,t),\qquad
 \frac{N-1}{r}z(r,t)\longrightarrow (N-1)u_{rr}(0,t).
\]
The flux identity
\begin{equation}\label{eq:flux-detail}
 (r^{N-1}z)_r=r^{N-1}H
\end{equation}
is valid classically after taking this limit. It follows that
\[
 H(0,t)=Nu_{rr}(0,t),\qquad
 \frac{u_r(r,t)}r\longrightarrow \frac{H(0,t)}N.
\]
Thus a bound for $H$ gives a bound for the full Cartesian Hessian only after the radial factor $r^{-1}$ has been treated through \eqref{eq:flux-detail}; simply bounding $u_r$ does not suffice.

The velocity relation also gives a useful one-sided estimate without differentiating the equation:
\begin{equation}\label{eq:velocity-H}
 H^\alpha=\frac{u_t}{\sqrt{1+u_r^2}}-b.
\end{equation}
If $0<m\le u_t\le V$ and $|u_r|\le P$, then
\[
 \left(m/S_P-b\right)^{1/\alpha}\le H\le(V-b)^{1/\alpha},
 \qquad S_P=(1+P^2)^{1/2},
\]
which is stronger than the lower bound $H>B$ when $m>0$. Inserting this in \eqref{eq:flux-detail} gives the centre estimate and then \eqref{eq:rrbound}. This is the reason a genuinely upper and lower estimate for $u_t$, rather than only an estimate for $u$, is the natural second-order input.

For the continuation theorem one also needs compatibility of the oblique boundary operator with the regularity class. In the present notation the boundary operator is
\[
 \mathcal G(r,t,u,u_r)=u_r-u\quad(r=1),
\]
and its derivative with respect to $u_r$ is $1$, so it is uniformly oblique. However, the standard continuation theorem still requires the coefficient bounds displayed in Proposition~\ref{prop:continuation} and a boundary Schauder estimate. This estimate is available once the finite-time curvature and velocity bounds are closed; it is not a substitute for proving them.

\subsection{Why the finite-time and uniform estimates cannot be merged}

Suppose, for example, that Proposition~\ref{prop:continuation} has been proved with a number $V_T$ for each $T$. Applying it to the translated function $v_s(r,t)=u(r,s+t)-u(0,s)$ on $[-L,L]$ gives a constant depending on $s+L$. It therefore gives no bound as $s\to\infty$. To obtain a constant independent of $s$, one must prove \eqref{eq:PV} directly, or prove a time-independent trapping estimate that implies it through \eqref{eq:velocity-H}.

This distinction is also visible in the Schauder step. On a finite cylinder the parabolicity ratio is
\[
 \frac{\Lambda_T}{\lambda_T}
 \le (1+P_T^2)\frac{\max\{B^{\alpha-1},K_T^{\alpha-1}\}}
 {\min\{B^{\alpha-1},K_T^{\alpha-1}\}},
\]
which may grow with $T$. Interior compactness of time translates requires the analogous ratio with $P_{R,\tau}$ and $V_{R,\tau}$, independent of the translation time. This is exactly the estimate recorded in Proposition~\ref{prop:interior}.

\section{The radial zero-number step}

For two positive-curvature solutions $u_1,u_2$, let $v=u_1-u_2$. Writing $p_i=(u_i)_r$, $S_i=(1+p_i^2)^{1/2}$, and using the mean-value theorem separately in $H$ and $S$, one obtains
\begin{equation}\label{eq:diff}
 v_t=A\left(v_{rr}+\frac{N-1}{r}v_r\right)+D(r,t)v_r,qquad A>0,
\end{equation}
on compact classical cylinders; $D(r,t)=O(r)$ at the centre. The exact formula is
\[
 A=S_1\left(\alpha\int_0^1(H_2+\sigma(H_1-H_2))^{\alpha-1}\dd\sigma\right)
 \left(\int_0^1(1+(p_2+\sigma(p_1-p_2))^2)^{-3/2}\dd\sigma\right).
\]
More explicitly, put
\[
 k=\int_0^1(1+(p_2+\sigma(p_1-p_2))^2)^{-3/2}\dd\sigma,
 \quad E=\alpha\int_0^1(H_2+\sigma(H_1-H_2))^{\alpha-1}\dd\sigma.
\]
Then $A=S_1Ek$ and
\begin{equation}\label{eq:diff-coeff}
 D=S_1E k_r+(H_2^\alpha+b)\frac{p_1+p_2}{S_1+S_2}.
\end{equation}
Indeed, $p_1/S_1-p_2/S_2=kv_r$, so
\[
 H_1-H_2=k\left(v_{rr}+\frac{N-1}{r}v_r\right)+k_rv_r.
\]
Now split $F[u_1]-F[u_2]$ as $S_1(H_1^\alpha-H_2^\alpha)+(H_2^\alpha+b)(S_1-S_2)$. This proves both formulas without assuming that an interpolation of the full jets stays on the positive-curvature branch. Smooth radial regularity gives $p_i=O(r)$ and $k_r=O(r)$, hence $D=O(r)$.

We use the zero-number theorem of Angenent and Chen--Pol\'a\v cik: for a nontrivial radial solution of \eqref{eq:diff}, the number of spatial zeros is finite and nonincreasing on intervals where the endpoint sign/boundary hypotheses are valid, and drops at an interior multiple zero. This imported theorem is all that is needed abstractly.

There is a boundary defect which must not be suppressed. If $L_h=\Phi_h+c(h)t+A_h$ is a finite-slope translator and $v=u-L_h$, then
\begin{equation}\label{eq:defect}
 v_r(1,t)-v(1,t)=L_h(1,t)-h.
\end{equation}
Thus $v$ does not satisfy the homogeneous Robin condition unless $L_h(1,t)=h$, which is not preserved in time. The zero-number argument for a solution versus a translator therefore requires an effective interval with controlled endpoint sign, or a separate boundary-contact argument. It is not enough to quote the abstract zero-number theorem.

\begin{lemma}[Conditional crossing preservation]\label{lem:cross}
Suppose $v$ solves \eqref{eq:diff} on $[0,R]\times[t_0,t_1]$, has $v_r(0,t)=0$, has fixed endpoint signs $v(0,t)>0>v(R,t)$, and has exactly one simple zero at $t=t_0$. Then it has exactly one simple zero for every $t\in(t_0,t_1)$, with $v_r<0$ at that zero.
\end{lemma}
\begin{proof}
The endpoint signs give at least one zero. Zero-number monotonicity gives at most one. A multiple zero would cause a strict drop, contradicting the endpoint signs immediately afterwards. The orientation is positive to negative, hence the derivative is negative.
\end{proof}

\subsection{The normalized-slope equation}

The boundary defect \eqref{eq:defect} shows why the height difference with a finite-slope translator is not the right quantity for the uniform gradient estimate. We instead use the normalized slope
\[
 z=\frac{u_r}{\sqrt{1+u_r^2}}\in[0,1).
\]
This variable has a fixed boundary value for every finite-slope translator and converges to $1$ for the cup. A direct differentiation of \eqref{eq:pde} gives
\begin{equation}\label{eq:z-equation}
 z_t=a\left(z_{rr}+\frac{N-1}{r}z_r-\frac{N-1}{r^2}z\right)
      +(H^\alpha+b)z z_r,
 \qquad
 a=\frac{\alpha H^{\alpha-1}}{1+u_r^2},
\end{equation}
where
\[
 H=z_r+\frac{N-1}{r}z.
\]
For a translator $\Phi+c t$, the normalized slope $z_\Phi$ is stationary in
\eqref{eq:z-equation}. If $z_1,z_2$ correspond to two positive-curvature
solutions, the mean-value theorem gives, on every compact cylinder,
\begin{equation}\label{eq:z-difference}
 w_t=\widetilde a\left(w_{rr}+\frac{N-1}{r}w_r\right)
      +\widetilde d\,w_r+\widetilde c\,w,
      \qquad w=z_1-z_2,
\end{equation}
where $\widetilde a>0$ and the coefficients are continuous and bounded. Their
radial regularity at $r=0$ is the one required in the radial zero-number theorem.
Indeed, \eqref{eq:z-difference} follows by interpolating the three smooth
quantities $a$, $H$, and $z$ in \eqref{eq:z-equation}; no interpolation of the
full second-order jet is needed.

We shall use the following standard consequence of the radial zero-number
theorem of Angenent and Chen--Pol\'a\v cik.
\begin{lemma}[Radial sign persistence]\label{lem:z-sign}
Let $w$ solve \eqref{eq:z-difference} on $[0,1]\times[t_0,t_1]$, with the
smooth radial condition at $r=0$. Assume that $w(0,t)=0$, that $w(1,t)$ has a
strict fixed sign, and that $w(\cdot,t_0)$ has the same strict sign on
$(0,1)$. Then $w$ keeps that sign on $(0,1)$ for $t_0\le t\le t_1$.
\end{lemma}
\begin{proof}
The radial zero-number theorem applies to the lifted equation on the ball;
the apparent coefficients $r^{-1}$ and $r^{-2}$ in \eqref{eq:z-equation} are
regular after the radial expansion $z(r,t)=r\,\zeta(r^2,t)$. A change of sign
would produce a first interior multiple zero, or a zero entering through the
centre. The centre alternative is excluded by the centre version of the
radial theorem, while the strict sign at $r=1$ excludes entry through the
outer boundary. At an interior multiple zero the zero number drops strictly.
Before the first sign change the centre zero and the outer endpoint sign force
the same number of zeros, so such a drop is impossible. Hence the initial sign
persists.
\end{proof}

\begin{lemma}[Compatible squeezed profiles]\label{lem:squeezed}
Let $h>0$ be sufficiently large. There is $q=q(h)\in(0,1)$, with
$q\uparrow1$ as $h\uparrow\infty$, such that
\begin{equation}\label{eq:rhoq}
 \rho_q(r)=\Phi_*(q r)+M_q,\qquad
 M_q=q\Phi_*'(q)-\Phi_*(q),
\end{equation}
satisfies
\begin{gather}
 \rho_q'(0)=0,\qquad \rho_q'(1)=\rho_q(1)=q\Phi_*'(q)>h,\\
 0<z_h(r)<z_{\rho_q}(r)<z_*(r)\quad(0<r<1).
\end{gather}
Here $z_h=h/\sqrt{1+h^2}$ at $r=1$ and $z_h$ denotes the normalized slope
of $\Phi_h$ in the interior.
\end{lemma}
\begin{proof}
The first line follows directly from \eqref{eq:rhoq}. The strict derivative
ordering is the profile-ODE selection property recorded in
Assumption~\ref{ass:profiles}; applying the increasing map
$p\mapsto p/\sqrt{1+p^2}$ gives the normalized-slope ordering. The common
zero at $r=0$ is harmless because all three profiles have positive second
derivative there, so the inequalities are strict for $r>0$.
\end{proof}

\begin{proposition}[Uniform interior gradient estimate from zero numbers]
\label{prop:uniform-gradient-special}
Let $u$ be the solution with the compatible special datum $u(\cdot,0)=\rho_q$
from Lemma~\ref{lem:squeezed}, and suppose the finite-time estimates of
Section~3 hold. Then, for every $R<1$,
\begin{equation}\label{eq:uniform-gradient-special}
 \sup_{t\ge0}\sup_{0\le r\le R}u_r(r,t)\le \Phi_*'(R)<\infty.
\end{equation}
More precisely, for the corresponding $h$,
\[
 \Phi_h'(r)<u_r(r,t)<\Phi_*'(r),
 \qquad 0<r<1, t\ge0.
\]
Consequently, the estimate is independent of the time translation used in
the compactness argument.
\end{proposition}
\begin{proof}
By Lemma~\ref{lem:squeezed}, the two differences
\[
 w^-(r,t)=z(r,t)-z_h(r),\qquad
 w^+(r,t)=z(r,t)-z_*(r)
\]
have respectively positive and negative signs at $t=0$ for $0<r<1$.
Both satisfy an equation of the form \eqref{eq:z-difference}. At the centre
they vanish. At the outer boundary, $z_*(1^-)=1$ and $z(1,t)<1$ on every
finite time interval. More precisely, on $[0,T]$ the finite-time gradient
bound gives $\sup_{t\le T}z(1-\varepsilon,t)<1$; choosing $\varepsilon>0$
sufficiently small makes $z(1-\varepsilon,t)<z_*(1-\varepsilon)$ for all
$t\le T$. Thus the upper comparison is carried out first on
$[0,1-\varepsilon]$. Moreover $u_t>0$ by Lemma~\ref{lem:C1},
and therefore $u(1,t)\ge u(1,0)=\rho_q(1)>h$; hence
\[
 w^-(1,t)=\frac{u(1,t)}{\sqrt{1+u(1,t)^2}}
          -\frac{h}{\sqrt{1+h^2}}>0.
\]
Lemma~\ref{lem:z-sign} now gives $w^->0$ on $(0,1)$ and $w^+<0$ on
$(0,1-\varepsilon]$ for $0\le t\le T$. Since $R<1$ and $T$ are arbitrary,
the estimate holds on every $[0,R]\times[0,\infty)$. Inverting the increasing map
$p\mapsto p/\sqrt{1+p^2}$ yields the asserted slope inequalities. The upper
bound on $[0,R]$ is therefore independent of time.
\end{proof}

\subsection{Transfer from the special datum to general data}

For arbitrary compatible convex data, the preceding proof is used after a
standard one-crossing transfer. One first chooses vertical translates of a
special profile and of the cup so that the initial height differences have
one prescribed zero. The difference of two solutions satisfies \eqref{eq:diff}
with the homogeneous Robin condition, while the normalized-slope difference
satisfies \eqref{eq:z-difference}. The radial zero-number theorem preserves
the one-crossing pattern: a first loss of the upper (respectively lower)
slope inequality creates an interior multiple zero, and the strict sign at
$r=1$ prevents the zero from escaping through the boundary. Thus the same
argument gives, for every $R<1$,
\begin{equation}\label{eq:transfer-A}
 \sup_{t\ge t_0}\sup_{0\le r\le R}u_r(r,t)
 \le \Phi_*'(R),
\end{equation}
provided the initial one-crossing condition is satisfied at some $t_0$.
The derivative estimate is the substantive part of this transfer; height
ordering alone would not imply it. In a general-data theorem,
\eqref{eq:transfer-A} (and the analogous time-derivative bound) must be
verified as an explicit hypothesis or by carrying out this sliding argument
for the chosen initial class. This is precisely the role of the detailed
zero-number argument in \cite[Section~4]{YuanZhao2022} and
\cite[Section~4]{LouWangYuan2021}; Chen--Pol\'a\v cik supplies the radial
zero-number theorem used in the present higher-dimensional step.

The important point for the convergence proof is now explicit: once
\eqref{eq:transfer-A} holds, it supplies the missing time-independent
gradient part of \eqref{eq:PV}. Together with the time-independent velocity
bound in \eqref{eq:PV}, Lemma~\ref{lem:Hessian} then gives time-independent
interior bounds for $u_{rr}$, $u_r/r$, $H$, and $u_t$. This is the estimate
that cannot be obtained by applying the finite-time bounds with $T=s+L$.

\section{Interior compactness and convergence}

\begin{proposition}[Time-uniform interior estimates]\label{prop:interior}
Assume \eqref{eq:PV}. For every $R<1$ and $\tau>0$,
\begin{equation}\label{eq:interior}
 \sup_{t\ge\tau}\sup_{0\le r\le R}(|u_r|+|u_{rr}|+H)<\infty,
\end{equation}
and the Cartesian equation is uniformly parabolic there, with constants independent of time.
\end{proposition}
\begin{proof}
Apply Lemma~\ref{lem:Hessian} with $P=P_{R,\tau}$ and $V=V_{R,\tau}$. The eigenvalues of the Cartesian Hessian of a radial function are $u_{rr}$ and $u_r/r$, with multiplicity $N-1$, so the full Hessian is bounded, including at $r=0$. Since $B\le H\le(V_{R,\tau}-b)^{1/\alpha}$ and $|u_r|\le P_{R,\tau}$, the eigenvalues of the matrix derivative of the fully nonlinear operator are bounded above and below by positive time-independent constants.\end{proof}

\begin{proposition}[Compactness of normalized translates]\label{prop:compact}
Under \eqref{eq:PV}, for $R<1$ and $L<\infty$, the family
\[
 v_s(r,t)=u(r,s+t)-u(0,s),\qquad |t|\le L,
\]
is bounded in $C^{2+\eta,1+\eta/2}([0,R]\times[-L,L])$ for all sufficiently large $s$.
\end{proposition}
\begin{proof}
Choose $R<R'<1$ and $s\ge\tau+L+2$. Integrating the uniform spatial and time derivative bounds gives
\[
 |v_s(r,t)|\le R'P_{R',\tau}+(L+2)V_{R',\tau}
 \quad(0\le r\le R',\ |t|\le L+2).
\]
Proposition~\ref{prop:interior} gives uniform gradient, Hessian and ellipticity bounds on the enlarged cylinder. For fixed gradient the operator is concave in the Hessian for $0<\alpha\le1$, and convex for $\alpha\ge1$, on the positive-curvature branch. The interior regularity theorem for smooth uniformly parabolic convex or concave fully nonlinear equations, followed by Schauder regularity, therefore applies; see \cite[Chapter XIV]{Lieberman1996}. All structural constants on the bounded positive-curvature range and both cylinder separations are independent of $s$. The radial centre is an interior point in Cartesian variables.
\end{proof}

\begin{proof}[Proof of Theorem~\ref{thm:main}]
Take $s_j\to\infty$ and extract $v_{s_j}\to U$ in $C^{2,1}_{\loc}$ by Proposition~\ref{prop:compact}. Fix $h$, $R<1$ and $L<\infty$. The slope trapping gives, for all large $j$,
\[
 \Phi_h'(r)\le (v_{s_j})_r(r,t)\le\Phi_*'(r)\quad(0\le r\le R, |t|\le L).
\]
Let $j\to\infty$, then $h\to\infty$. By $\Phi_h\to\Phi_*$ in $C^2([0,R])$,
\[
 U_r(r,t)=\Phi_*'(r).
\]
Therefore $U(r,t)=\Phi_*(r)+A(t)$. Substitution gives $A'=c_*$. Since $U(0,0)=0$, $A(t)=c_*t$. Every subsequence has the same limit, proving the full convergence.
\end{proof}

\subsection{Identification of the entire limit}

The limit $U$ is defined on the whole time line because, for each fixed $L$, the interval $[-L,L]$ is contained in the translated domain once $s$ is large. No expansion in powers of $1/s$ is used. For fixed $h$, the trapping estimate first gives
\[
 \Phi_h'(r)\le U_r(r,t)\le\Phi_*'(r).
\]
Letting $h\to\infty$ and using the $C^2$ convergence of profiles on $[0,R]$ gives $U_r=\Phi_*'$ simultaneously for $t$ in every compact interval. Integration in $r$ gives
\[
 U(r,t)-U(0,t)=\Phi_*(r).
\]
Substitution into the equation yields $U_t=c_*$. The normalization $U(0,0)=0$ fixes the additive constant.

This also explains why the full family, rather than only a selected sequence, converges. If convergence failed on a compact cylinder, one could select a sequence staying a fixed positive distance from the claimed limit. Proposition~\ref{prop:compact} would produce a further convergent subsequence, while the preceding identification forces its limit to be the claimed translator, a contradiction. Finally,
\[
 (v_s)_t(r,t)=u_t(r,s+t),
\]
so the $C^{2,1}$ convergence gives $u_t\to c_*$. Equation \eqref{eq:velocity-H} and the flux identity then identify the limiting $H$, $u_r/r$, and $u_{rr}$ on every compact subset of $r<1$.

\begin{proposition}[Average speed only]\label{prop:speed}
Under Assumption~\ref{ass:profiles}, global smooth existence, and the upper comparison, finite-slope lower barriers (when their initial-order conditions hold) imply $u(r,t)/t\to c_*$ uniformly in $r$. This assertion alone does not imply \eqref{eq:PV} or \eqref{eq:slope-trap}.
\end{proposition}
\begin{proof}
The upper comparison gives the limsup $\le c_*$. For fixed $h$, compare after a sufficiently late time with an admissible lower translate, giving liminf $\ge c(h)$. Let $h\to\infty$.
\end{proof}

\section{Stationary cap}

For $B<N$, put $a=B/N$ and $R_0=a^{-1}$. The stationary profile with $H=B$ is
\begin{equation}\label{eq:cap}
 U_A(r)=A+R_0-\sqrt{R_0^2-r^2},\qquad
 U_A'(r)=\frac{ar}{\sqrt{1-a^2r^2}}.
\end{equation}
The constant $A$ is selected by the boundary law. To derive the profile, set $z=U_A'/\sqrt{1+U_A'^2}$ in $(r^{N-1}z)'=Br^{N-1}$. Radial regularity eliminates the integration constant, giving $z=Br/N$ and hence \eqref{eq:cap}. For $g(s)=s$, the unique constant is determined by $U_A(1)=1/\sqrt{R_0^2-1}$.

\begin{proposition}[Linear instability of the stationary cap]
If $g\in C^1$, $U_A'(1)=g(U_A(1))$, and $\gamma=g'(U_A(1))>0$, the radial linearization at $U_A$ has a positive principal eigenvalue.
\end{proposition}
\begin{proof}
Put $k(r)=\sqrt{1-a^2r^2}$ and $K_0=\alpha B^{\alpha-1}$. At the stationary cap, $H^\alpha+b=0$, so the derivative of the factor $\sqrt{1+u_r^2}$ contributes zero. Linearization gives
\[
 Lv=\frac{K_0}{k(r)}r^{1-N}\bigl(r^{N-1}k(r)^3v_r\bigr)_r,
 \qquad v_r(0)=0,\quad v_r(1)=\gamma v(1).
\]
Writing $L=m^{-1}(Pv_r)_r$, integration by parts yields the Rayleigh quotient for the largest radial eigenvalue:
\[
 \lambda_1=\sup_{v\ne0}\frac{P(1)g'(U_A(1))v(1)^2-\int_0^1P(r)v_r^2\dd r}{\int_0^1m(r)v^2\dd r},
\]
where $P(r)=r^{N-1}(1-a^2r^2)^{3/2}$ and $m(r)=r^{N-1}\sqrt{1-a^2r^2}/(\alpha B^{\alpha-1})$. The supremum is over the radial weighted $H^1$ form domain. The Robin condition is natural in this variational formulation; trial functions need not satisfy it pointwise. Taking $v\equiv1$ gives
\[
 \lambda_1\ge\frac{P(1)\gamma}{\int_0^1m(r)\dd r}>0.
\]
The associated eigenmode grows like $e^{\lambda_1t}$ in the linearized problem.
\end{proof}
Thus convergence of all small positive data to this cap is not an automatic companion theorem. A threshold or restricted-data convergence statement would require a separate invariant-region or spectral argument. No such nonlinear small-data result is claimed here.

\section{Power-law Robin conditions}

Consider $u_r(1,t)=u(1,t)^\theta$, $\theta>0$, on positive boundary values. The finite-slope lower residual is
\[
 h-\bigl(\Phi_h(1)+A_h+c(h)t\bigr)^\theta\le0
\]
provided $A_h\ge h^{1/\theta}-\Phi_h(1)$. A fixed finite-slope profile is not a global upper barrier: its residual eventually becomes negative for every $\theta>0$, including $\theta=1$. The singular cup upper comparison is instead the first-contact argument of Lemma~\ref{lem:upper}; it must not be written as a classical boundary residual ``$+\infty-g(W)$''.

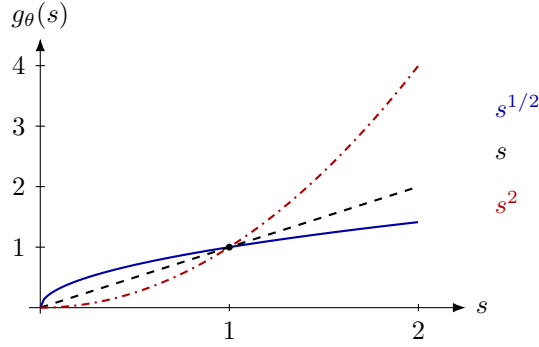
\begin{figure}[htbp]
\centering
\begin{tikzpicture}[>=Latex,font=\small,x=2.5cm,y=.8cm]
 \draw[->] (-.06,0)--(2.25,0) node[right] {$s$}; \draw[->] (0,-.1)--(0,4.45) node[above] {$g_\theta(s)$};
 \foreach \x in {1,2}{\draw (\x,.06)--(\x,-.06) node[below] {$\x$};}
 \foreach \y in {1,2,3,4}{\draw (.025,\y)--(-.025,\y) node[left] {$\y$};}
 \draw[blue!65!black,thick,domain=0:2,samples=100] plot (\x,{sqrt(\x)});
 \draw[black,thick,dashed] (0,0)--(2,2);
 \draw[red!70!black,thick,dash dot,domain=0:2,samples=100] plot (\x,{\x*\x});
 \fill (1,1) circle (1.3pt);
 \node[blue!65!black,anchor=west] at (2.35,3.35) {$s^{1/2}$};
 \node[anchor=west] at (2.35,2.55) {$s$};
 \node[red!70!black,anchor=west] at (2.35,1.75) {$s^2$};
\end{tikzpicture}
\caption{The three power-law boundary responses on $0\le s\le2$. The sublinear law has unbounded derivative at zero, so its linearized boundary coefficient needs a positive height lower bound.}
\label{fig:power}
\end{figure}

For $w=u_t$ and for $v=u_1-u_2$, differentiation gives respectively
\begin{align}
 w_r(1,t)&=\theta u(1,t)^{\theta-1}w(1,t),\label{eq:power-w}\\
 v_r(1,t)&=\gamma(t)v(1,t),\qquad
 \gamma(t)=\theta\int_0^1(u_2(1,t)+\sigma v(1,t))^{\theta-1}\dd\sigma.\label{eq:power-v}
\end{align}
If $0<m\le u_i(1,t)\le M$, then $0<\gamma(t)\le\theta\max\{m^{\theta-1},M^{\theta-1}\}$. Thus, on finite intervals, the power law changes the boundary coefficient but not the interior difference equation. The conditional convergence theorem remains valid if the independent inputs \eqref{eq:PV} and \eqref{eq:slope-trap} are established for this boundary law.

\begin{remark}[General increasing boundary laws]
For $g\in C^2(\R)$, $g'>0$, and $g(s)\to+\infty$, choose $S_h$ with $g(S_h)\ge h$ for the lower comparison. The singular upper comparison is unchanged. The difference boundary coefficient is the divided difference of $g$ and is bounded on bounded height ranges. This verifies the comparison identities, not an unconditional global-existence theorem.
\end{remark}

\begin{remark}
The finite-time part of the paper is complete conditional on the upper velocity estimate \eqref{eq:finiteV}: it gives $u$, $u_r$, $u_{rr}$ and uniform parabolicity with constants depending on $T$. The zero-number proposition in Section~4 supplies the time-uniform interior gradient estimate for the squeezed special data and identifies the precise transfer statement needed for arbitrary data. The convergence part is then complete conditional on the time-uniform inputs \eqref{eq:PV} and \eqref{eq:slope-trap}; it gives time-uniform Hessian estimates, compactness, and unique identification of all translated limits. The remaining issue is sharply localized in the Robin velocity estimate and, for a fully general initial-data theorem, in verifying the one-crossing transfer hypothesis in \eqref{eq:transfer-A}. This is the higher-dimensional analogue of Sections~4--5 of \cite{YuanZhao2022} and Section~4 of \cite{LouWangYuan2021}.
\end{remark}

\vskip 8mm
\begin{center}
{\large\bf Statements and Declarations}
\end{center}

\noindent
{\bf Funding}. No funding was received for conducting
this study.

\noindent
{\bf Competing interests}. The authors declare no
competing interests.

\noindent
{\bf Data availability}. No datasets were generated or
analysed during the current study.

\noindent
{\bf Author contributions}. The first and second authors ..., the third author conceived the study, developed the proofs; the forth author wrote the manuscript.

\noindent
{\bf Use of generative artificial intelligence}. During the preparation of this manuscript, the authors
used ChatGPT to assist with drafting and language editing. The
author subsequently reviewed and revised the
entire manuscript and takes full responsibility for
its content.

\end{document}